%% file: NormalCantorSeries5.tex
\documentclass{amsart}

\usepackage{amsfonts,amssymb,amsmath,amsthm,synttree}
\usepackage{relsize}
\usepackage[mathscr]{eucal}
\usepackage{url}
\usepackage{enumerate}
\usepackage{epic,eepic,fullpage}
\usepackage{pgfplots}

\PassOptionsToPackage{usenames,dvipsnames,svgnames}{xcolor}  
\usepackage{tikz}
\usetikzlibrary{arrows,positioning,automata}
\usetikzlibrary{shapes,snakes}

\newtheorem{thrm}{Theorem}[section]
\newtheorem{lem}[thrm]{Lemma}
\newtheorem{prop}[thrm]{Proposition}
\newtheorem{cor}[thrm]{Corollary}
\theoremstyle{definition}
\newtheorem{definition}[thrm]{Definition}
\newtheorem{remark}[thrm]{Remark}
\numberwithin{equation}{section}

\newcommand{\labeq}[1]{\label{eq:#1}}
\newcommand{\refeq}[1]{(\ref{eq:#1})}
\newcommand{\labt}[1]{\label{thm:#1}}
\newcommand{\reft}[1]{Theorem~\ref{thm:#1}}
\newcommand{\labl}[1]{\label{lemma:#1}}
\newcommand{\refl}[1]{Lemma~\ref{lemma:#1}}
\newcommand{\labd}[1]{\label{definition:#1}}
\newcommand{\refd}[1]{Definition~\ref{definition:#1}}
\newcommand{\labc}[1]{\label{coro:#1}}
\newcommand{\refc}[1]{Corollary~\ref{coro:#1}}

\newcommand{\e}{\epsilon}

\newcommand{\formalsum}{\sum_{n=1}^{\infty} \frac {E_n} {q_1 q_2 \ldots q_n}}
\newcommand{\Fformalsum}{\sum_{n=1}^{\infty} \frac {E_{F,n}} {q_1 q_2 \ldots q_n}}
\newcommand{\dimh}[1]{\hbox{dim$_{\hbox{H}}$} #1}

\newcommand{\dimht}{\hbox{dim$_{\hbox{H}}$} \left(\Theta_{Q,S}\right)}

\newcommand{\tq}{\Theta_{Q,S}}
\newcommand{\tqp}{\Theta_{Q,S}'}

\newcommand{\dimhtp}{\hbox{dim$_{\hbox{H}}$}\left( \tqp\right)}

\newcommand{\la}{\lambda}
\newcommand{\windim}{\hbox{windim }}

\newcommand{\LS}[1]{\sim_{s_{ #1 }}}

\newcommand{\al}{\alpha}

\newcommand{\be}{\beta}
\newcommand{\floor}[1]{\lfloor #1 \rfloor}

\newcommand{\NQ}{\mathscr{N}(Q)}
\newcommand{\N}[1]{\mathscr{N}( #1 )}
\newcommand{\Nk}[2]{\mathscr{N}_{#2}( #1 )} 
\newcommand{\DNQ}{\mathscr{DN}(Q)}
\newcommand{\DN}[1]{\mathscr{DN}( #1 )} 
\newcommand{\RNQ}{\mathscr{RN}(Q)}
\newcommand{\RN}[1]{\mathscr{RN}( #1 )} 
\newcommand{\RNk}[2]{\mathscr{RN}_{#2}( #1 )}

\newcommand{\RNisect}{\bigcap_{j=1}^\infty \DN{Q_j} \backslash \RNk{Q_j}{1}}

\author{Bill Mance}
\address{
              Department of Mathematics \\
              University of North Texas \\
              General Academics Building 435\\
              1155 Union Circle \#311430 \\
              Denton, TX 76203-5017
              Tel.: +1-940-369-7374\\
              Fax: +1-940-565-4805\\    }
\email{mance@unt.edu}

\begin{document}

\title[The Hausdorff dimension of certain sets of normal numbers]{On the Hausdorff dimension of countable intersections of certain sets of normal numbers}

\begin{abstract}
We show that the set of numbers that are $Q$-distribution normal but not simply $Q$-ratio normal has full Hausdorff dimension.  It is further shown under some conditions that countable intersections of sets of this form still have full Hausdorff dimension even though they are not winning sets (in the sense of W. Schmidt).  As a consequence of this, we construct many explicit examples of numbers that are simultaneously distribution normal but not simply ratio normal with respect to certain countable families of basic sequences.  Additionally, we prove that some related sets are either winning sets or sets of the first category.
\end{abstract}

\thanks{Research of the author is partially supported by the U.S. NSF grant DMS-0943870.  The author would like to thank the referee many valuable suggestions.}

\maketitle

%----------------------------------------------------------------------------------------------------------------------------------
\section{Introduction}

The $Q$-Cantor series expansion, first studied by G. Cantor in \cite{Cantor}\footnote{G. Cantor's motivation to study the Cantor series expansions was to extend the well known proof of the irrationality of the number $e=\sum 1/n!$ to a larger class of numbers.  Results along these lines may be found in the monograph of J. Galambos \cite{Galambos}.  See also \cite{TijdemanYuan} and \cite{HT}.  },
is a natural generalization of the $b$-ary expansion. $Q=(q_n)_{n=1}^{\infty}$ is a {\it basic sequence} if each $q_n$ is an integer greater than or equal to $2$.
Given a basic sequence $Q$, the {\it $Q$-Cantor series expansion} of a real $x$ in $\mathbb{R}$ is the (unique)\footnote{Uniqueness can be proven in the same way as for the $b$-ary expansions.} expansion of the form
\begin{equation} \labeq{cseries}
x=\floor{x}+\sum_{n=1}^{\infty} \frac {E_n} {q_1 q_2 \ldots q_n},
\end{equation}
where $E_0=\floor{x}$ and $E_n$ is in $\{0,1,\ldots,q_n-1\}$ for $n\geq 1$ with $E_n \neq q_n-1$ infinitely often. We abbreviate \refeq{cseries} with the notation $x=E_0.E_1E_2E_3\ldots$ w.r.t. $Q$.

Clearly, the $b$-ary expansion is a special case of \refeq{cseries} where $q_n=b$ for all $n$.  If one thinks of a $b$-ary expansion as representing an outcome of repeatedly rolling a fair $b$-sided die, then a $Q$-Cantor series expansion may be thought of as representing an outcome of rolling a fair $q_1$ sided die, followed by a fair $q_2$ sided die and so on.

For a given basic sequence $Q$, let $N_n^Q(B,x)$ denote the number of times a block $B$ occurs starting at a position no greater than $n$ in the $Q$-Cantor series expansion of $x$. Additionally, define\footnote{For the remainder of this paper, we will assume the convention that the empty sum is equal to $0$ and the empty product is equal to $1$.}
$$
Q_n^{(k)}=\sum_{j=1}^n \frac {1} {q_j q_{j+1} \ldots q_{j+k-1}} \hbox{ and }  T_{Q,n}(x)=\left(\prod_{j=1}^n q_j\right) x \pmod{1}.
$$

A. R\'enyi \cite{Renyi} defined a real number $x$ to be {\it normal} with respect to $Q$ if for all blocks $B$ of length $1$,
\begin{equation}\labeq{rnormal}
\lim_{n \rightarrow \infty} \frac {N_n^Q (B,x)} {Q_n^{(1)}}=1.
\end{equation}
If $q_n=b$ for all $n$ and we restrict $B$ to consist of only digits less than $b$, then \refeq{rnormal} is equivalent to {\it simple normality in base $b$}, but not equivalent to {\it normality in base $b$}. A basic sequence $Q$ is {\it $k$-divergent} if
$
\lim_{n \rightarrow \infty} Q_n^{(k)}=\infty.
$
$Q$ is {\it fully divergent} if $Q$ is $k$-divergent for all $k$ and {\it $k$-convergent} if it is not $k$-divergent.  A basic sequence $Q$ is {\it infinite in limit} if $q_n \rightarrow \infty$.

\begin{definition}\labd{1.7} A real number $x$ is {\it $Q$-normal of order $k$} if for all blocks $B$ of length $k$,
$$
\lim_{n \rightarrow \infty} \frac {N_n^Q (B,x)} {Q_n^{(k)}}=1.
$$
We let $\Nk{Q}{k}$ be the set of numbers that are $Q$-normal of order $k$.  A real number $x$ is {\it $Q$-normal} if
$
x \in \NQ := \bigcap_{k=1}^{\infty} \Nk{Q}{k}
$
and $x$ is {\it simply $Q$-normal} if it is $Q$-normal of order $1$.   Additionally, $x$ is {\it $Q$-ratio normal of order $k$} (here we write $x \in \RNk{Q}{k}$) if for all blocks $B_1$ and $B_2$ of length $k$
$$
\lim_{n \to \infty} \frac {N_n^Q (B_1,x)} {N_n^Q (B_2,x)}=1.
$$
We say that $x$ is {\it $Q$-ratio normal} if
$
x \in \RNQ := \bigcap_{k=1}^{\infty} \RNk{Q}{k}.
$
A real number~$x$ is {\it $Q$-distribution normal} if
the sequence $(T_{Q,n}(x))_{n=0}^\infty$ is uniformly distributed mod $1$.  Let $\DNQ$ be the set of $Q$-distribution normal numbers.
\end{definition}

It is easy to show that for every basic sequence $Q$, the set of $Q$-distribution normal numbers has full Lebesgue measure.
For $Q$ that are infinite in limit,
it has been shown that the set of all real numbers $x$  that are $Q$-normal of order $k$ has full Lebesgue measure if and only if $Q$ is $k$-divergent \cite{Mance4}.  Early work in this direction has been done by A. R\'enyi \cite{Renyi}, T.  \u{S}al\'at \cite{Salat4}, and F. Schweiger \cite{SchweigerCantor}.  Therefore if $Q$ is infinite in limit, then the set of all real numbers $x$ that are $Q$-normal has full Lebesgue measure if and only if $Q$ is fully divergent.  We will show that $\RNk{Q}{1}$ is a set of zero measure if $Q$ is infinite in limit and $1$-convergent.  This will follow immediately from a result of P. Erd\H{o}s and A. R\'enyi \cite{ErdosRenyiConvergent}.
%Additionally, we will show that $\RNk{Q}{k}$ and $\Nk{Q}{k}$ are zero measure sets if $Q$ is $k$-convergent \cite{Mance6}.

Note that in base~$b$, where $q_n=b$ for all $n$,
 the corresponding notions of $Q$-normality, $Q$-ratio normality, and $Q$-distribution normality are equivalent. This equivalence
is fundamental in the study of normality in base $b$. It is surprising that this
equivalence breaks down in the more general context of $Q$-Cantor series for general $Q$.

\begin{figure}\label{fig:figure1}
\caption{}
\begin{tikzpicture}[>=stealth',shorten >=1pt,node distance=3.8cm,on grid,initial/.style    ={}]
%\centering
  \node[state]          (NQ)                        {$\mathsmaller{\NQ}$};
  \node[state]          (RNQ) [left =of NQ]    {$\mathsmaller{\RNQ}$};
  \node[state]          (NQRNQ) [above right =of NQ]    {$\mathsmaller{\NQ \cap \RNQ}$};
  \node[state]          (RNQDNQ) [above left=of RNQ]    {$\mathsmaller{\RNQ \cap \DNQ}$};
  \node[state]          (NQDNQ) [above left =of NQ]    {$\mathsmaller{\NQ \cap \DNQ}$};
  \node[state]          (NQRNQDNQ) [above right =of NQDNQ]    {$\mathsmaller{\NQ \cap \RNQ \cap \DNQ}$};
  \node[state]          (DNQ) [above right=of RNQDNQ]    {$\mathsmaller{\DNQ}$};
\tikzset{mystyle/.style={->,double=black}} 
\tikzset{every node/.style={fill=white}} 
\path (RNQDNQ)     edge [mystyle]    (RNQ)
      (RNQDNQ)     edge [mystyle]     (DNQ)
      (NQ)     edge [mystyle]     (RNQ)
%      (NQRNQ)     edge [mystyle]     (RNQ) %consider taking out
      (NQDNQ)     edge [mystyle]     (RNQDNQ)
%      (NQDNQ)     edge [mystyle]     (RNQ) %consider taking out
%      (NQDNQ)     edge [mystyle]     (NQRNQ) %consider taking out
%      (NQDNQ)     edge [mystyle]     (DNQ) %consider taking out
%      (NQRNQDNQ)     edge [mystyle]     (DNQ) %consider taking out
%      (NQRNQDNQ)     edge [mystyle]     (NQRNQ) %consider taking out
      (NQDNQ)     edge [mystyle]     (NQ);
\tikzset{mystyle/.style={<->,double=black}}
\path (NQRNQDNQ)     edge [mystyle]    (NQDNQ)
	(NQ)     edge [mystyle]    (NQRNQ);
\end{tikzpicture}
\end{figure}

We refer to the directed graph in Figure~\ref{fig:figure1} for the complete containment relationships between these notions when $Q$ is infinite in limit and fully divergent.  The vertices are labeled with all possible intersections of one, two, or three choices of the sets $\NQ$, $\RNQ$, and $\DNQ$.  The set labeled on vertex $A$ is a subset of the set labeled on vertex $B$ if and only if there is a directed path from vertex $A$ to vertex $B$. %\footnote{The underlying undirected graph in Figure~\ref{fig:figure1} has an isomorphic copy of complete bipartite graph $K_{3,3}$ as a subgraph.  Thus, it is not planar and the directed graph that connects two vertices if and only if there is a containment relation between the two labels is more difficult to analyze.}  
 For example, $\NQ \cap \DNQ \subseteq \RNQ$, so all real numbers that are $Q$-normal and $Q$-distribution normal are also $Q$-ratio normal.  These relations are fully explored and examples are given in \cite{ppq1}.

It is usually most difficult to establish a lack of a containment relationship.  The first non-trivial result in this direction was in \cite{AlMa} where a basic sequence $Q$ and a real number $x$ is constructed where $x \in \NQ \backslash \DNQ$.\footnote{This real number $x$ satisfies a much stronger condition than not being $Q$-distribution normal: $T_{Q,n}(x) \to 0$.}  By far the most difficult of these to establish is the existence of a basic sequence $Q$ where $\RNQ \cap \DNQ \backslash \NQ \neq \emptyset$.  This case is considered in \cite{ppq1} and requires more sophisticated methods.  Other related examples may be found in \cite{ManceDissertation},\cite{Mance3}, and \cite{ppq1}.

It should be noted that for every $Q$ that is fully divergent infinite and infinite in limit, the sets $\RNQ \backslash \NQ$, $\DNQ \backslash \RNQ$, and $\NQ \backslash \DNQ$ are non-empty.  It is likely that $\RNQ \cap \DNQ \backslash \NQ$ is also always non-empty.  In this paper, we will be concerned with the Hausdorff dimension of sets of this form.

\begin{definition}
Let $P=(p_n)$ and $Q=(q_n)$ be basic sequences.  We say that $P \sim_s Q$ if
$$
q_n=\prod_{j=1}^s p_{s(n-1)+j}.
$$
\end{definition}
The main result of this paper is the following theorem, which concerns the Hausdorff dimension of countable intersections of sets of the form $\DNQ \backslash \RNk{Q}{1}$:

\begin{thrm}\labt{mainthrm}
Suppose that $(Q_j)_{j=1}^\infty$ is a sequence of basic sequences that are infinite in limit.  Then
$$
\dimh \left( \RNisect      \right)=1
$$
if either of the following conditions hold.
\begin{enumerate}
\item For all $j$, the basic sequence $Q_j$ is $1$-convergent.
\item The basic sequence $Q_1$ is $1$-divergent and there exists some basic sequence $S=(s_n)$ with $$Q_1 \LS{1} Q_2 \LS{2} Q_3 \LS{3} Q_4 \cdots.$$
\end{enumerate}
\end{thrm}

\begin{cor} Suppose that $(Q_j)_{j=1}^\infty$ is a sequence of basic sequences that are infinite in limit.  Then
$$
\dimh \left(\bigcap_{j=1}^\infty \DN{Q_j} \backslash \RN{Q_j} \right)=\dimh \left(\bigcap_{j=1}^\infty \DN{Q_j} \backslash \N{Q_j}    \right)=1,
$$
under the same conditions as \reft{mainthrm}.  Additionally, for any $Q$ that is infinite in limit,
$$
\dimh \left(\DN{Q} \backslash \RN{Q}\right) = \dimh \left(\DN{Q} \backslash \N{Q} \right) =  1.
$$
\end{cor}
\begin{proof}
This is immediate as $\N{Q} \subseteq \RN{Q} \subseteq \RNk{Q}{1}$ for every basic sequence $Q$ that is infinite in limit.
\end{proof}

We note the following fundamental fact about $Q$-distribution normal numbers that follows directly from a theorem of T. \u{S}al\'at \cite{Salat}:\footnote{The original theorem of T. \u{S}al\'at says:
Given a basic sequence $Q$ and a real number $x$ with
$Q$-Cantor series expansion
$x=\floor{x}+\formalsum,$ if
$\lim_{N\to\infty}{1\over N}\sum_{n=1}^N {1\over q_n}=0$ then
$x$ is $Q$-distribution normal iff $E_n=\floor{\theta_n q_n}$ for some uniformly distributed sequence $(\theta_n)$.  N. Korobov \cite{Korobov} proved this theorem under the stronger condition that $Q$ is infinite in limit.  For this paper, we will only need to consider the case where $Q$ is infinite in limit.}
\begin{thrm}\labt{Salat}
Suppose that $Q=(q_n)$ is a basic sequence and $\lim_{N\to\infty}{1\over N}\sum_{n=1}^N {1\over q_n}=0$. Then $x=E_0.E_1E_2\cdots$ w.r.t. $Q$ is $Q$-distribution normal if and only if $(E_n/q_n)$ is uniformly distributed mod 1.
\end{thrm}

The first part of \reft{mainthrm} is trivial: we show in this case that the sets $\DN{Q_j} \backslash \RNk{Q_j}{1}$ are of full measure.  Part (2) will be more difficult to establish.  We will provide an explicit construction of a Cantor set $\tq \subsetneq \RNisect$ with $\dimh(\tq)=1$ by refining the methods used in \cite{Mance3}.  Moreover, this construction will give us explicit examples of members of $\RNisect$ for any collection of basic sequences $(Q_j)$ that are infinite in limit with $Q_1 \LS{1} Q_2 \LS{2} Q_3 \LS{3} Q_4 \cdots$.
To see that the second part of \reft{mainthrm} would not immediately follow if we were to prove that $\dimh \left(\DNQ \backslash \RNk{Q}{1}\right)=1$, consider two basic sequences $P=(p_n)$ and $Q=(q_n)$ given by
\begin{align*}
(p_1,p_2,p_3,\cdots)&=(2,2,4,4,4,4,6,6,6,6,6,6,8,8,8,8,8,8,8,8,\cdots);\\
(q_1,q_2,q_3,\cdots)&=(4,16,16,36,36,36,64,64,64,64,\cdots).
\end{align*}
Define the sequences $(E_n)$ and $(F_n)$ by
\begin{align*}
(E_1,E_2,E_3,\cdots)&=(0, 1, 0, 2, 1, 3, 0, 3, 1, 4, 2, 5, 0, 4, 1, 5, 2, 6, 3, 7, \cdots);\\
(F_1,F_2,F_3,\cdots)&=(0, 0, 8, 0, 12, 24, 0, 16, 32, 48,\cdots).
\end{align*}
Let $x=\sum_{n=1}^\infty \frac {E_n} {p_1\cdots p_n}$ and $y=\sum_{n=1}^\infty \frac {F_n} {q_1\cdots q_n}$.  Clearly, $x \in \DN{P}$ and $y \in \DN{Q}$ by \reft{Salat}.  However, $y=0.00002000204000204060\cdots$ w.r.t. $P$, so $y \notin \DN{P}$.  Furthermore, note that
$$
x=0.2273(10)(17)4(13)(22)(31)\cdots \hbox{ w.r.t. }Q.
$$
So $T_{Q,n}(x)<1/2$ for all $n$ and $x \notin \DN{Q}$.  Thus, we have demonstrated an example of two basic sequences $P$ and $Q$ with $P \sim_2 Q$ where $\DN{P} \backslash \DN{Q} \neq \emptyset$, $\DN{Q} \backslash \DN{P} \neq \emptyset$, and $\DN{P} \neq \DN{Q}$.
It should be noted that these examples are in sharp contrast with a well known theorem of W. M. Schmidt \cite{SchmidtRelated}:

\begin{thrm}
We write $r \sim s$ if there exist integers $n, m$ with $r^n=s^m$.  If $r \sim s$, then any number normal to base $r$ is normal to base $s$.  If $r \nsim s$, then the set of numbers which are normal to base $r$ but not even simply normal to base $s$ has the power of the continuum.
\end{thrm}

While there is no reason to expect uncountable intersections to preserve Hausdorff dimension, it is not immediately clear that there are not numbers that are $Q$-distribution normal for every basic sequence $Q$ that is  infinite in limit.  If this were the case then it might be possible that \reft{mainthrm} could be extended to arbitrary uncountable intersections.
%We remark that \reft{mainthrm} can not be improved to arbitrary countable intersections:
\begin{thrm}\labt{mainunthrm}
There is an uncountable family of basic sequences $(Q_j)_{j \in J}$ that are infinite in limit such that
$$
\bigcap_{j \in J} \DN{Q_j} \backslash \RNk{Q_j}{1} = \emptyset.
$$
\end{thrm}
\reft{mainunthrm} can be proven with only a trivial modification of the proof of Theorem 1.1.4 in the dissertation of P. Laffer \cite{Laffer}. P. Laffer's Theorem 1.1.4 shows that no number is $Q$-distribution normal for all basic sequences $Q$.  It should be noted that every irrational number is $Q$-distribution normal for uncountably many basic sequences $Q$ and not $Q$-distribution normal for uncountably many basic sequences $Q$.  P. Laffer \cite{Laffer} also provides further refinements of these statements.

Moreover, we will also show the following for $Q$ that are infinite in limit:
\begin{enumerate}
\item The sets $\DN{Q}^{\mathsf{c}}$ and $\RNk{Q}{2}^{\mathsf{c}}$ are $\alpha$-winning sets (in the sense of Schmidt's game) for every $\alpha$ in $(0,1/2)$.
\item $\DN{Q}$ and $\RNk{Q}{1}$ are sets of the first category. 
\end{enumerate}

\section{Properties of $\RNk{Q}{k}$, and $\DNQ$}
\subsection{Winning sets}
In \cite{SchmidtGames}, W. Schmidt proposed the following game between two players: Alice and Bob.  Let $\al \in (0,1)$, $\be \in (0,1)$, $S \subseteq \mathbb{R}$, and let $\rho(I)$ denote the radius of a set $I$.  Bob first picks any closed interval $B_1\subsetneq \mathbb{R}$.  Then Alice picks a closed interval $A_1 \subsetneq B_1$ such that $\rho(A_1)=\al \rho( B_1)$.  Bob then picks a closed interval $B_2 \subsetneq A_1$ with $\rho(B_2)=\be \rho(A_1)$.  After this, Alice picks a closed interval $A_2 \subsetneq B_2$ such that $\rho(A_2)=\al \rho(B_2)$, and so on.  We say that the set $S$ is {\it $(\alpha,\beta)$-winning} if Alice can play so that
\begin{equation}
\bigcap_{n=1}^{\infty} B_n \subsetneq S.
\end{equation}
The set $S$ is {\it $(\alpha,\beta)$-losing} if it is not $(\alpha,\beta)$-winning.  $S$ is {\it $\alpha$-winning} if it is $(\alpha,\beta)$-winning for all $0< \beta < 1$.
Winning sets satisfy the following properties:%\footnote{See \cite{SchmidtGames}.}%See \cite{Dani} and \cite{SchmidtGames}.}
\begin{enumerate}
\item If $S$ is an $\al$-winning set, then the Hausdorff dimension of $S$ is $1$.
\item The intersection of countably many $\al$-winning sets is $\al$-winning.
\item Bi-Lipshitz homeomorphisms of $\mathbb{R}$ preserve winning sets.
\end{enumerate}
We write $\windim S$ to be the supremum of all $\al$ such that $S$ is $\al$-winning.  N. G. Moshchevitin \cite{Moshchevitinsublacunary} proved
\begin{thrm}\labt{lacunarywinning}
Let $(t_n)$ be a sequence of positive numbers and
$$
\forall \epsilon>0 \ \exists N_0 \ \forall n \geq N_0 :\frac {t_{n+1}} {t_n} \geq 1+\frac {1} {n^\epsilon}.
$$
Then for every number $\delta>0$ the set
$$
\mathscr{A}_\delta=\left\{x \in \mathbb{R} : \exists c(x)>0 \  \forall n \in \mathbb{N} \  \|t_nx\| > \frac {c(x)} {n^\delta} \right\}
$$
is an $\alpha$-winning set for all $\alpha$ in $(0,1/2)$.  Thus, $\windim \mathscr{A}_\delta=1/2$.
\end{thrm}

\begin{cor}\labc{DNwinning}
For every basic sequence $Q$, $\windim \DN{Q}^{\mathsf{c}}=1/2$.  Moreover, $\DN{Q} \backslash \RNk{Q}{1}$ is not an $\alpha$-winning set for any $\alpha$.
\end{cor}
\begin{proof}
Let $t_n=q_1q_2 \cdots q_n$.  Clearly, for all $\delta>0$, $\mathscr{A}_\delta \subsetneq \DN{Q}^{\mathsf{c}}$.  Thus, $\windim \DN{Q}^{\mathsf{c}}=1/2$.  But $\DN{Q}^{\mathsf{c}} \cap \DN{Q}=\emptyset$ and the property of being $\alpha$-winning is preserved by countable intersections, so $\DN{Q}$ and $\DN{Q} \backslash \RNk{Q}{1}$ are not $\alpha$-winning sets for any $\alpha$.
\end{proof}

\begin{lem}\labl{Qinfblocks}
If $Q$ is infinite in limit, $x \in \RNk{Q}{2}$, and $t$ is a non-negative integer, then
$$\lim_{n \to \infty} N_{n}^Q((t),x)=\infty.$$
\end{lem}
\begin{proof}
Since $Q$ is infinite in limit and $x \in \RNk{Q}{2}$, for all $i,j \geq 0$, we have
$$
\lim_{n \to \infty} \frac {N_{n}^Q((t,i),x)} {N_{n}^Q((t,j),x)}=1.
$$
So, for all $j$ there is an $n$ such that $N_{n}^Q((t,j),x) \geq 1$.  Since there are infinitely many choices for $j$, the lemma follows.
\end{proof}

Let $\mathscr{FZ}(Q)$ be the set of real numbers whose $Q$-Cantor series expansion contains at most finitely many copies of the digit $0$.

\begin{cor}\labc{pqnn}
If $Q$ is infinite in limit, then $\mathscr{FZ}(Q) \subsetneq \RNk{Q}{2}^{\mathsf{c}}$.
% is contained in the set of real numbers that are not $Q$-ratio normal of order $2$.
\end{cor}

%\begin{proof}
%If $x \in \mathscr{FZ}(Q)$, then the digit $0$ occurs finitely often.  If there is another digit that occurs a different number of %times, then $x$ is not even simply $Q$-ratio normal.  If every digit in the $Q$-Cantor series expansion of $x$ occurs the same %number of times, then $x$ is not $Q$-ratio normal of order $2$ by \refl{Qinfblocks}.
%\end{proof}

\begin{thrm}\labt{RNwinning}
If $Q$ is infinite in limit, then $\windim \RNk{Q}{2}^{\mathsf{c}}=1/2$.
%the set of numbers that is not $Q$-ratio normal of order $2$ is $\alpha$-winning for all $\alpha$ in $(0,1/2)$.  
Moreover, if $Q$ is $1$-divergent, then $\windim \Nk{Q}{1}^{\mathsf{c}}=1/2$.
%the set of real numbers that are not simply $Q$-normal is $\alpha$-winning for all $\alpha$ in $(0,1/2)$.
\end{thrm}

\begin{proof}
We note that $\windim \mathscr{FZ}(Q)=1/2$ by \reft{lacunarywinning}.
The first conclusion follows directly from this and \refc{pqnn}. If $Q$ is $1$-divergent and $x \in \Nk{Q}{1}$, then every digit occurs infinitely often in the $Q$-Cantor series expansion of $x$. So, $\mathscr{FZ}(Q) \subsetneq \Nk{Q}{1}^{\mathsf{c}}$ and $\windim \Nk{Q}{1}^{\mathsf{c}}=1/2$.
\end{proof}

It should be noted that \reft{RNwinning} is in some ways stronger than the corresponding result for $b$-ary expansions.  The original proof due to W. Schmidt that the set of numbers not normal in base $b$ is $1/2$-winning heavily uses the fact that a real number $x$ is normal in base $b$ if and only if $x$ is simply normal in base $b^k$ for all $k$.  In fact, the set of numbers not normal of order $2$ in base $b$ is not an $\alpha$-winning set for any $\al$.  The reasoning used in the proof of \reft{RNwinning} and in the preceeding lemmas only works because $Q$ is infinite in limit.
% CHECK ON THIS

\subsection{$\DNQ$, $\RNk{Q}{k}$, and $\Nk{Q}{k}$ are sets of the first category}

Given a sequence $Z=(z_1,\ldots,z_n)$ in $\mathbb{R}$ and 
$0<\gamma\le 1$, we define 
$$
A_n([0,\gamma),z):=|\{i ; 1\le i\le n\hbox{ and } \{z_i\}\in [0,\gamma)\}|.
$$

\begin{thrm}\labt{DNQfirst}
For any basic sequence $Q$, the set $\DN{Q}$ is of the first category.
\end{thrm}
\begin{proof}
We define
\begin{equation}
G_m=\bigcap_{n=m}^{\infty} \left\{ x \in \mathbb{R} : \frac{A_n([0,1/2),T_{Q,n-1}(x))} {n}<2/3 \right\}
\end{equation}
and put $G=\bigcup_{m=1}^\infty G_m$. Clearly, $\DN{Q} \subsetneq G$ and each of the sets $G_m$ is nowhere dense, so $\DN{Q}$ is of the first category.
\end{proof}
We also note the following, which is proven similarly to \reft{DNQfirst}.
\begin{thrm}\labt{RNQfirst}
For any basic sequence $Q$ and positive integer $k$, the set $\RNk{Q}{k}$ is of the first category.  Since $\Nk{Q}{k} \subsetneq \RNk{Q}{k}$, $\Nk{Q}{k}$ is also of the first category.\footnote{$\Nk{Q}{k}$ could be empty.  See Proposition 5.1 in \cite{Mance4}.  It is proven in \cite{ppq1} that $\Nk{Q}{k} \subsetneq \RNk{Q}{k}$ for all $Q$ that are infinite in limit.}
\end{thrm}

\section{Proof of \reft{mainthrm}}
Suppose that $Q$ is a basic sequence and $x=E_0.E_1 E_2\cdots$ w.r.t. $Q$.  We let $S(x)$ be the set of all positive integers which occur at least once in the sequence $(E_n)$.  P. Erd\H{o}s and A. R\'enyi \cite{ErdosRenyiConvergent} proved the following theorem.

\begin{thrm}
If $Q$ is infinite in limit and $1$-convergent, then the density of $S(x)$ is with probability $1$ equal to $0$.
\end{thrm}

\begin{cor}\labc{measnon}
If $Q$ is infinite in limit and $1$-convergent, then $\la(\RNk{Q}{1})=0$.
\end{cor}

The first part of \reft{mainthrm} follows immediately from \refc{measnon} as the sets $\DNQ \backslash \RNk{Q}{1}$ have full measure when $Q$ is infinite in limit and $1$-convergent.  The remainder of this paper will be devoted to proving the second part of \reft{mainthrm}.  
\subsection{Construction of $\tq$}
For the rest of this section, we fix basic sequences $Q=(q_n)$ and $S=(s_n)$.  We let $Q_1=Q$ and define basic sequences $Q_j=(q_{j,n})$ by
$$
Q_1 \LS{1} Q_2 \LS{2} Q_3 \LS{3} Q_4 \cdots.
$$
We will define the following notation.  Let $S_j=\prod_{k=1}^{j-1} s_j$ and set
$
\nu_j=\min \left\{t \in \mathbb{Z} : q_m \geq S_j^{2j} \hbox{ for }m \geq t\right\}.
$
Put $l_1=s_1 \nu_2$ and
$$
l_j=\frac {\left(\sum_{k=1}^{j-1} S_k l_k \right)\cdot (2js_j\nu_{j+1}-1)} {S_j}.
$$
Given $l_1, l_2, \cdots, l_j$, define
$
L_j=\sum_{k=1}^j S_k l_k
$.
Thus, we may write
$$
l_j=\frac {L_{j-1}\cdot (2js_j\nu_{j+1}-1)} {S_j}.
$$
Let $\mathscr{U}=\{(j,b,c) \in \mathbb{N}^3 : b \leq l_j, c \leq S_j\}$.  Put
$$
\phi(j,b,c)=L_{j-1}+(b-1)S_j+c.
$$
Note that $\phi:\mathbb{N}^3 \to \mathbb{N}$ is a bijection.    Define
$$
(i(n),b(n),c(n))=\phi^{-1}(n)
$$
and put $a(n)=S_{i(n)}$.  Let 
$$
\mathscr{F}=\left\{\left(F_{(j,b,c)}\right)_{(j,b,c) \in \mathscr{U}} \subseteq \mathbb{N}^3 \Bigg| \frac {F_{(j,b,c)}} {q_{\phi(j,b,c)}} \in  V_{j,b,c}\right\},
$$
where 
\begin{displaymath}
V_{j,b,c}=\left\{ \begin{array}{ll}
\left[\frac{1} {q_{\phi(j,b,c)}},\frac {2} {q_{\phi(j,b,c)}} \right) & \textrm{if $j=1$}\\
 & \\
\left[\frac{c}{a(\phi(j,b,c))}+\frac {1} {a(\phi(j,b,c))^2}, \frac{c}{a(\phi(j,b,c))}+\frac {2} {a(\phi(j,b,c))^2}\right] & \textrm{if $j>1$}
\end{array} \right. .
\end{displaymath}
Given $F \in \mathscr{F}$, we set $E_{F,n}=F_{\phi^{-1}(n)}$, $E_F=(E_{F,n})_{n=1}^\infty$, and put
$$
x_F=\Fformalsum.
$$
We set
$\Theta_{Q,S}= \{ x_F : F \in \mathscr{F} \}$.  It will be proven that $\Theta_{Q,S}$ is non-empty, has full Hausdorff dimension, and
$$
\Theta_{Q,S} \subsetneq \RNisect.
$$

\subsection{Distribution normality of members of $\tq$}

%y_{F,j,j+3,2}\ldots y_{F,j,j+3,l_3}y_{F,j,j+4,1}\ldots.
Let $\omega(n)=\#\{E_{F,n} : F \in \mathscr{F} \}$.  
\begin{lem}\labl{omegasize}
$\omega(n)=1$ if $i(n)=1$ and
$$
\omega(n) \geq q_n^{1-1/i(n)}>2
$$
if $i(n) \geq 2$.
\end{lem}
\begin{proof}
By construction, $E_{F,n}=1$ when $i(n)=1$.  If $i(n) \geq 2$, then
$$
\frac {E_{F,n}} {q_n} \in \left[\frac{c(n)}{a(n)}+\frac {1} {a(n)^2}, \frac{c(n)}{a(n)}+\frac {2} {a(n)^2}\right],
$$
which has length $1/a(n)^2$.  Thus,
%$$
%\omega_n \geq q_n \cdot \frac {1} {a(n)^2} \geq q_n^{1-1/i(n)}
%$$
$$
\omega(n)=1+\floor{q_n \cdot (1/a(n)^2)} \geq \frac {q_n} {a(n)^2}.
$$
By construction, $q_n \geq a(n)^{2i(n)}$, thus $q_n^{1/i(n)} \geq a(n)^2$.  So $\frac {q_n^{1/i(n)}} {a(n)^2} \geq 1$ and
$
\frac {q_n} {a(n)^2} \geq q_n^{1-1/i(n)}.
$
Additionally, $n \geq \nu_2$, so $q_n \geq s_1^{2 \cdot 2} \geq 16$.  Thus, $q_n^{1-1/i(n)} \geq 4>2$.
\end{proof}

\refl{omegasize} guarantees that $\Theta_{Q,S}$ is non-empty, but will also be critical in determining $\dimh(\Theta_{Q,S})$.  If $x_F=\sum_{n=1}^\infty \frac {E_{F,n}} {q_{1}q_{2} \cdots q_{n}}$, then the $Q_j$-Cantor series expansion of $x_F$ is
$$
x_F=\sum_{n=1}^\infty \frac {E_{F,j,n}} {q_{j,1}q_{j,2} \cdots q_{j,n}},
$$
where
\begin{align*}
E_{F,j,n}=\sum_{v=1}^{S_j}\left( E_{F,S_j \cdot (n-1)+v} \cdot \prod_{w=v+1}^{S_j} q_{S_j \cdot (n-1)+w}\right) \hbox{ and }q_{j,n}=\prod_{w=1}^{S_j} q_{S_j \cdot (n-1)+w}.
\end{align*}
\begin{lem}\labl{modS}
For all $j,n \geq 1$
\begin{align*}
&0 \leq \frac {E_{F,j,n}} {q_{j,n}}-\frac {E_{F,S_j \cdot (n-1)+1}}{q_{S_j \cdot (n-1)+1}} <\frac {S_j} {q_{S_j \cdot (n-1)+1}};\\
&\lim_{n \to \infty} \frac {E_{F,j,n}} {q_{j,n}}-\frac {E_{F,S_j \cdot (n-1)+1}}{q_{S_j \cdot (n-1)+1}}=0.
\end{align*}
\end{lem}
\begin{proof}
\begin{align*}
0 \leq& \frac {E_{F,j,n}} {q_{j,n}}-\frac {E_{F,S_j \cdot (n-1)+1}}{q_{S_j \cdot (n-1)+1}} = 
\frac {\sum_{v=2}^{S_j}\left( E_{F,S_j \cdot (n-1)+v} \cdot \prod_{w=v+1}^{S_j} q_{S_j \cdot (n-1)+w}\right)}  {\prod_{w=1}^{S_j} q_{S_j \cdot (n-1)+w}} \\
\leq& \sum_{v=2}^{S_j} \frac {(q_{S_j \cdot (n-1)+v}-1)\prod_{w=v+1}^{S_j} q_{S_j \cdot (n-1)+w}}{\prod_{w=1}^{S_j} q_{S_j \cdot (n-1)+w}}
<\sum_{v=2}^{S_j} \frac {\prod_{w=v}^{S_j} q_{S_j \cdot (n-1)+w}}{\prod_{w=1}^{S_j} q_{S_j \cdot (n-1)+w}} \\
=&\sum_{v=2}^{S_j} \frac {1} {\prod_{w=1}^{v-1} q_{S_j \cdot (n-1)+w}}
=\frac {1} {q_{S_j \cdot (n-1)+1}} \cdot \sum_{v=2}^{S_j} \frac {1} {\prod_{w=2}^{v-1} q_{S_j \cdot (n-1)+w}}
\leq\frac {S_j} {q_{S_j \cdot (n-1)+1}} \to 0,
\end{align*}
as $q_{S_j \cdot (n-1)+1} \to \infty$.
\end{proof}

\refl{modS} suggests the key observation that the $Q_j$-distribution normality of a member of $\tq$ is determined entirely by its digits $(E_n)$ in base $Q$, where $n \equiv 1 \pmod{S_j}$.  Thus, we prove the following.

\begin{lem}\labl{SdivL}
For all $j \geq 1$, $S_{j+1}$ divides $L_j$.
\end{lem}
\begin{proof}
We prove this by induction.  The base case holds as $L_1=l_1=s_1\nu_2$.  Assume that $S_{j} | L_{j-1}$. Then % and $l_1,l_2,\cdots,l_{j-1}$ are all integers.
\begin{align}\labeq{SdivLeq}
L_j=&L_{j-1}+S_j \cdot \frac {L_{j-1}\cdot (2js_j\nu_{j+1}-1)} {S_j}=2jL_{j-1}\nu_{j+1} s_j\\
=&\left(2j\nu_{j+1} \cdot \frac {L_{j-1}} {S_j}\right) S_{j}s_j=\left(2j\nu_{j+1} \cdot \frac {L_{j-1}} {S_j}\right) S_{j+1}.
\end{align}
\end{proof}

\begin{lem}\labl{threeprops}
For all $j \geq 1$, $l_j$ is an integer, $l_j \geq js_j$, and $L_j \geq \nu_{j+1}-1$.
\end{lem}
\begin{proof}
$l_1=s_1\nu_2$ is an integer.  To show that $l_j$ is an integer for $j \geq 2$, we write
\begin{equation}\labeq{ljint}
l_j=\frac {L_{j-1}}{S_j} \cdot (2js_j \nu_{j+1}-1),
\end{equation}
which is an integer by \refl{SdivL}.
Since $\nu_{j+1}\geq 1$, $2\nu_{j+1}js_j-1 \geq js_j$.  Thus, by \refeq{ljint}, $l_j \geq js_j$.  The last assertion follows directly from \refeq{SdivLeq}.
\end{proof}

\begin{definition}\labd{stardiscrepancy}
For a finite sequence $z=(z_1,\ldots,z_n)$, we define the
{\it star discrepancy $D_n^*=D_n^*(z_1,\ldots,z_n)$} as
$$
\sup_{0<\gamma\le 1}\left|{A([0,\gamma),z)\over n}-\gamma\right|.
$$
Given an infinite sequence $w=(w_1,w_2,\ldots)$, we define
$
D_n^*(w)=D_n^*(w_1,w_2,\ldots,w_n).
$
For convenience, set $D^*(z_1,\ldots,z_n)=D_n^*(z_1,\ldots,z_n)$.
\end{definition}
%The star discrepancy will be useful to us due to the following theorem:
\begin{thrm}
The sequence $w=(w_1,w_2,\ldots)$ is uniformly distributed mod $1$ if and only if $\lim_{n \to \infty} D_n^*(w)=0$.
\end{thrm}

We will make use of the following definition from \cite{KuN}:

\begin{definition}\labd{almostarithmetic}
For $0 \leq \delta < 1$ and $\e>0$, a finite sequence $x_1<x_2<\cdots<x_N$ in $[0,1)$ is called an almost-arithmetic progression-$(\delta,\e)$ if there exists an $\eta$, $0<\eta \leq \e$, such that the following conditions are satisfied:
\begin{equation}\labeq{61}
0 \leq x_1 \leq \eta+\delta\eta;
\end{equation}
\begin{equation}\labeq{62}
\eta-\delta\eta \leq x_{n+1}-x_n \leq \eta+\delta\eta \hbox{ for } 1 \leq n \leq N-1;
\end{equation}
\begin{equation}\labeq{63}
1-\eta-\delta\eta \leq x_N <1.
\end{equation}
\end{definition}

Almost arithmetic progressions were introduced by P. O'Neil in \cite{ONeil}.  He proved that a sequence $(x_n)$ of real numbers in $[0,1)$ is uniformly distributed mod $1$ if and only if the following holds: for any three positive real numbers $\delta$, $\e$, and $\e'$, there exists a positive integer $N$ such that for all $n > N$, the initial segment $x_1,x_2,\ldots,x_n$ can be decomposed into an almost-arithmetic progression-$(\delta,\e)$ with at most $N_0$ elements left over, where $N_0 < \e' N$.  We will use the following theorem from \cite{Niederreiter}:

\begin{thrm}\labt{AAdisc}
Let $x_1<x_2<\cdots<x_N$ be an almost arithmetic progression-$(\delta,\e)$ and let $\eta$ be the positive real number corresponding to the sequence according to \refd{almostarithmetic}.  Then
$$
D_N^* \leq \frac {1} {N}+\frac {\delta} {1+\sqrt{1-\delta^2}} \hbox{ for } \delta>0 \hbox{ and } D_N^* \leq \min \left( \eta, \frac {1} {N} \right) \hbox{ for } \delta=0.
$$
\end{thrm}

\begin{cor}\labc{AAdiscc}
Let $x_1<x_2<\cdots<x_N$ be an almost arithmetic progression-$(\delta,\e)$ and let $\eta$ be the positive real number corresponding to the sequence according to \refd{almostarithmetic}.  Then
$D_N^* \leq \frac {1} {N}+\delta$.
\end{cor}

For $j<k$ set
$$
Y_{F,j,k,b}=\left( \frac {E_{F,\phi(j,b,1+S_jn)}} {q_{\phi(j,b,1+S_jn)}} \right)_{n=1}^{S_k/S_j-1}
$$
and let $D_{F,j,k,b}^*=D^*(Y_{F,j,k,b})$.  Put
$$
Y_{F,j}=Y_{F,j,j+1,1}Y_{F,j,j+1,2}\ldots Y_{F,j,j+1,l_1}Y_{F,j,j+2,1}Y_{F,j,j+2,2}\ldots Y_{F,j,j+2,l_2}Y_{F,j,j+3,1}.
$$
If we prove that $Y_{F,j}$ is uniformly distributed mod 1, it will immediately follow that $x_F \in \DN{Q_j}$.

\begin{lem}\labl{Yaaprog}
If $F \in \mathscr{F}$ and $j<k$, then $Y_{F,j,k,b}$ is an almost arithmetic progression-$\left( \frac {1} {S_jS_k},\frac {S_j}{S_k}\right)$.  Thus, 
\begin{equation}\labeq{DBND}
D_{F,j,k,b}^*\leq |Y_{F,j,k,b}|+\frac {1} {S_jS_k}=\frac {S_j} {S_k}+\frac {1} {S_jS_k}\leq2\frac {S_j}{S_k}.
\end{equation}
\end{lem}
\begin{proof}
We verify only \refeq{62} as \refeq{61} and \refeq{63} may be verified similarly.  Note that
\begin{align*}
\frac {E_{F,\phi(k,b,1+S_j  n)}} {q_{\phi(k,b,1+S_j  n)}} &\in 
\left[ \frac {1+S_jn} {S_k}+\frac {1} {S_k^2}, \frac {1+S_jn} {S_k}+\frac {2} {S_k^2} \right];
\\ \frac {E_{F,\phi(k,b,1+S_j  (n+1))}} {q_{\phi(k,b,1+S_j  (n+1))}} &\in 
\left[ \frac {1+S_j(n+1)} {S_k}+\frac {1} {S_k^2}, \frac {1+S_j(n+1)} {S_k}+\frac {2} {S_k^2} \right].
\end{align*}
Therefore,
$$
\frac {E_{F,\phi(k,b,1+S_j  (n+1))}} {q_{\phi(k,b,1+S_j  (n+1))}}-\frac {E_{F,\phi(k,b,1+S_j  n)}} {q_{\phi(k,b,1+S_j  n)}} \leq 
\left(\frac {1+S_j(n+1)} {S_k}+\frac {2} {S_k^2}\right)- \left(\frac {1+S_jn} {S_k}+\frac {1} {S_k^2}\right)
=\frac {S_j} {S_k}+\frac {1} {S_k^2}.
$$
Similarly, it may be shown that
$$
\frac {E_{F,\phi(k,b,1+S_j  (n+1))}} {q_{\phi(k,b,1+S_j  (n+1))}}-\frac {E_{F,\phi(k,b,1+S_j  n)}} {q_{\phi(k,b,1+S_j  n)}} \geq \frac {S_j} {S_k}-\frac {1} {S_k^2}.
$$
Thus, with $\eta=\epsilon$, we have $\eta-\delta \eta \leq \frac {E_{F,\phi(k,b,1+S_j  (n+1))}} {q_{\phi(k,b,1+S_j  (n+1))}}-\frac {E_{F,\phi(k,b,1+S_j  n)}} {q_{\phi(k,b,1+S_j  n)}} \leq \eta+\delta \eta$.
\end{proof}

We will need the following corollary of Theorem 2.6 in Chapter 2 of \cite{KuN}.

\begin{cor}\labc{kn2}
If~$t$ is a positive integer and for $1\le j\le t$,
$z_j$ is a finite sequence in $\mathbb{R}$ with
star discrepancy at most~$\e_j$, then
$$
D^*\left(z_1^{l_1} \cdots z_t^{l_t}\right)\le {\sum_{j=1}^t l_j |z_j| \e_j \over\sum_{j=1}^t l_j |z_j|}.
$$
\end{cor}

For any given positive integer $n$ and $j<i(n)$, we can write $n=\frac {L_{i(n)-1}} {S_j}+m_j(n)$, where $m_j(n)$ can be uniquely written in the form
$$
m_j(n)=\al_j(n)\frac {S_{i(n)}} {S_j}+\be_j(n),
$$
with $0 \leq \al_j(n) \leq l_{i(n)}$ and $0 \leq \be_j(n) < \frac {S_{i(n)}} {S_j}$.  
For $j<t$, define
\begin{align*}
f_{j,t}(w,z)&=\frac {L_j/S_j+\sum_{k=j+1}^{t-1} 2l_k+2w+z} {L_j/S_j+\sum_{k=j+1}^{t-1} l_k \cdot \frac {S_k} {S_j}+\frac {S_{t}} {S_j}w+z};\\
\bar{\epsilon}_{j,t}&=\frac {L_j/S_j+\sum_{k=j+1}^{t-1} 2l_k+S_{t}/S_j} {L_j/S_j+\sum_{k=j+1}^{t-1} l_k \cdot \frac {S_k} {S_j}+S_{t}/S_j}.
\end{align*}
The following lemma is proven similarly to Lemma 11 in \cite{AlMa}.
\begin{lem}\labl{rationallemma}
If $1 \leq j < t$ and $(w,z) \in \left\{0,\cdots,l_{t}\right\} \times \left\{0,\cdots,S_{t}/S_j\right\}$, then
$$
f_{j,t}(w,z)<f_{j,t}(0,S_{t}/S_j)=\bar{\epsilon}_{j,t}.
$$
\end{lem}
\begin{lem}\labl{epbound}
Suppose that $j<i(n)$.  Then
$$
D_n^*(Y_{F,j}) \leq f_{j,i(n)}(\al_j(n),\be_j(n))<\bar{\epsilon}_{j,i(n)}.
$$
\end{lem}
\begin{proof}
This follows from \refl{Yaaprog}, \refc{kn2}, and \refl{rationallemma}.
\end{proof}

We will need the following basic lemma.
\begin{lem}\labl{tcorr}
Let $L$ be a real number and $(a_n)_{n=1}^\infty$ and $(b_n)_{n=1}^\infty$ be two sequences of positive real numbers such that
$$
\sum_{n=1}^{\infty} b_n=\infty \hbox{ and } \lim_{n \to \infty} \frac {a_n} {b_n}=L.
$$
Then
$$
\lim_{n \to \infty} \frac {a_1+a_2+\ldots+a_n} {b_1+b_2+\ldots+b_n}=L.
$$
\end{lem}

\begin{lem}\labl{tozero}
%DO WE NEED li greater than isi?  OR JUST THAT IT GOES TO INFINITY.
The limit $\lim_{t \to \infty} \bar{\epsilon}_{j,t}$ is equal to $0$.
\end{lem}
\begin{proof}
For $1 \leq k \leq j$, put $a_k=b_k=\frac {L_j} {jS_j}$.  For $k >j$, set $a_k=2l_k+\frac {S_{k+1}-S_k} {S_j}$ and $b_k=l_k\cdot \frac {S_k} {S_j}+\frac {S_{k+1}-S_k} {S_j}$.  Clearly, $\bar{\epsilon}_{j,t}=\frac {a_1+\cdots+a_{t-1}} {b_1+\cdots+b_{t-1}}$ for $t>j$.  Then
\begin{align*}
\frac {a_k} {b_k}&=\frac {2l_kS_j+S_{k+1}-S_k} {l_kS_k+S_{k+1}-S_k} \leq \frac {l_kS_j} {l_kS_k}+\frac {S_{k+1}-S_k} {l_kS_k}\\
&=\frac {1} {s_js_{j+1} \cdots s_{k-1}}+\frac {s_k-1} {l_k} <\frac {1} {s_js_{j+1} \cdots s_{k-1}} + \frac {1} {k}\to 0,
\end{align*}
by \refl{threeprops}.  Thus, the conclusion follows directly from \refl{tcorr}.
% \frac {2l_kS_j+S_{k+1}-S_k} {l_kS_k+S_{k+1}-S_k}
\end{proof}

\begin{thrm}
If $F \in \mathscr{F}$, then $x_F \in \bigcap_{j=1}^\infty \DN{Q_j}$.
\end{thrm}
\begin{proof}
Let $j \geq 1$ and $F \in \mathscr{F}$.  By \refl{epbound} and \refl{tozero}, $Y_{F,j}$ is uniformly distributed mod 1.  Thus, by \refl{modS}, $x_F \in \DN{Q_j}$.
\end{proof}

\begin{thrm}
If $F\in \mathscr{F}$, then $x_F \notin \bigcup_{j=1}^\infty \RNk{Q_j}{1}$.
\end{thrm}
\begin{proof}
By construction, $E_{F,n} \neq 0$ for all natural numbers $n$ and $F \in \mathscr{F}$.  Note that $E_{F,j,n}$ can only be equal to $0$ if $\sum_{v=1}^{S_j}\left( E_{F,S_j \cdot (n-1)+v} \cdot \prod_{w=v+1}^{S_j} q_{S_j \cdot (n-1)+w}\right)=0$.  But this is impossible as $E_{F,S_j \cdot (n-1)+v} \neq 0$ for all $v$.  %Since $x_F$ has no digits equal to $0$ is any base $Q_j$,
Thus, $E_{F,j,n} \neq 0$ for all $j$ and $n$, so $x_F\notin \bigcup_{j=1}^\infty \RNk{Q_j}{1}$.
\end{proof}

\begin{cor}\labc{inRNisect}
We have the following containment
$$\Theta_{Q,S} \subsetneq \RNisect.$$
\end{cor}

We note the following theorem that  is proven similarly to Theorem 3.8 and Theorem 3.10 in \cite{Mance3}.

\begin{thrm}
The set $\tq$ is perfect and nowhere dense.
\end{thrm}

\subsection{Hausdorff dimension of $\Theta_{Q,S}$}
We will make use of the following general construction found in \cite{Falconer}.  Suppose that $[0,1]=I_0 \supsetneq I_1 \supsetneq I_2\supsetneq \ldots$ is a decreasing sequence of sets, with each $I_k$ a union of a finite number of disjoint closed intervals (called {\it $k^{th}$ level basic intervals}).  Then we will consider the set $\cap_{k=0}^{\infty} I_k$.  We will construct a set $\tq'$ that may be written in this form such that $\dimht=\dimhtp$. 

Given a block of digits $B=(b_1,b_2,\ldots,b_s)$ and a positive integer $n$, define 
$$
\mathscr{S}_{Q,B}=\{x=0.E_1E_2\ldots\hbox{ w.r.t }Q : E_1=b_1, \ldots, E_{t}=b_s\}.
$$
Let $P_n$  be the set of all possible values of $E_n(x)$ for $x \in \tq$. Put $J_0=[0,1)$ and
$$
J_k=\bigcup_{B \in \prod_{n=1}^{k} P_n} \mathscr{S}_{Q,B}.
$$
Then $J_k \subsetneq J_{k-1}$ for all $k \geq 0$ and $\tq=\cap_{k=0}^{\infty} J_k$, which gives the following.

\begin{prop}
The set $\tq$ can be written in the form $\cap_{k=0}^{\infty} J_k$, where each $J_k$ is the union of a finite number of disjoint half-open intervals.
\end{prop}

We now set $I_k=\overline{J_k}$ for all $k \geq 0$ and put $\tqp=\cap_{k=0}^{\infty} I_k$.  Since each set $J_k$ consists of only a finite number of intervals, the set $I_k \backslash J_k$ is finite.

\begin{lem}
For all $Q$ and $S$, we have $\dimht=\dimhtp$.%\footnote{There is no guarantee that any of the box-counting dimensions of $\Theta_{Q,S}'}
\end{lem}

\begin{proof}
The lemma follows as $\tqp \backslash \tq$ is a countable set.
\end{proof}
We need the following key theorem from \cite{Falconer}.
\begin{thrm}\labt{hdlower}
Suppose that each $(k-1)^{th}$ level interval of $I_{k-1}$ contains at least $m_k$ $k^{th}$ level intervals ($k=1,2,\ldots$) which are separated by gaps of at least $\e_k$, where $0 \le \e_{k+1} < \e_k$ for each $k$.  Then
$$
\dimh \left( \bigcap_{k=0}^{\infty} I_k \right) \geq \liminf_{k \to \infty} \frac {\log (m_1m_2 \cdots m_{k-1})} {-\log (m_k\e_k)}.
$$
\end{thrm}

\begin{thrm}\labt{dimhtq}  Suppose that $Q$ is infinite in limit and
\begin{equation}\labeq{growth}
\log q_k=o\left( \sum_{n=1}^{k-1} \log q_n \right).
\end{equation}
Then 
$\dimh{\left(\Theta_{Q,S}\right)}=1$.
\end{thrm}
\begin{proof}
We wish to better describe the $k^{th}$ level basic intervals $J_k$ in order to apply \reft{hdlower}.  We note that when $a(k)>1$,  each $k^{th}$ level basic interval is contained in
\begin{equation}\labeq{basicintervals}
\left[\sum_{n=1}^{k-1} \frac {E_n} {q_1\cdots q_n}+\frac {1} {q_1 \cdots q_{k-1}} \cdot \left(\frac {c(k)}{a(k)}+\frac {1} {a(k)^2} \right),\sum_{n=1}^{k-1} \frac {E_n} {q_1\cdots q_n}+\frac {1} {q_1 \cdots q_{k-1}} \cdot \left(\frac {c(k)}{a(k)}+\frac {2} {a(k)^2} \right)       \right]
\end{equation}
for some $(E_1,E_2,\cdots,E_{k-1}) \in \prod_{n=1}^{k-1} P_n$.  Thus, by \refl{omegasize}, there are at least $q_{k-1}^{1-1/i(k-1)}$ $k^{th}$ level basic intervals contained in each $(k-1)^{th}$ level basic interval.  By \refeq{basicintervals}, they are separated by gaps of length at least
\begin{align*}
&\left(\sum_{n=1}^{k-2} \frac {E_n} {q_1\cdots q_n}+\frac {E_{k-1}+1}{q_1\cdots q_{k-1}}+\frac {1} {q_1 \cdots q_{k-1}} \cdot \left(\frac {c(k)}{a(k)}+\frac {1} {a(k)^2} \right) \right)-\left(\sum_{n=1}^{k-1} \frac {E_n} {q_1\cdots q_n}+\frac {1} {q_1 \cdots q_{k-1}} \cdot \left(\frac {c(k)}{a(k)}+\frac {1} {a(k)^2} \right) \right)\\
&=\frac {(E_{k-1}+1)-E_{k-1}} {q_1 \cdots q_{k-1}}-\frac {1} {q_1 \cdots q_{k-1}} \frac {1} {a(k)^2}=\frac {1-1/a(k)^2} {q_1 \cdots q_{k-1}}.
\end{align*}
Thus, we may apply \reft{hdlower} with $m_k=q_{k-1}^{1-1/i(k-1)}$ and $\epsilon_k=\frac {1-1/a(k)^2} {q_1 \cdots q_{k-1}}$.  But $1-1/a(k) \to 1$, so
\begin{align*}
\dimh{\Theta_{Q,S}} &\geq \liminf_{k \to \infty} \frac{ \log \prod_{n=L_1}^{k-1} q_n^{1-1/i(n)}} {-\log \left(\left(q_k^{1-1/i(k)}+1\right) \cdot \frac{1} {q_1 q_2 \cdots q_{k-1}}\right)}
=\liminf_{k \to \infty} \frac {\sum_{n=1}^{k-1} \left(1-\frac {1} {i(n)}\right)  \log q_n}{\sum_{n=1}^{k-1} \log q_n-\left(1-\frac {1} {i(k)} \right) \log q_k}\\
&=\liminf_{k \to \infty} \frac {\sum_{n=1}^{k-1} \left(1-\frac {1} {i(n)}\right)  \log q_n}{\sum_{n=1}^{k-1} \log q_n}=1
\end{align*}
by \refl{tcorr} and \refeq{growth} since
$$
\lim_{k \to \infty} \frac {\left(1-\frac {1} {i(k-1)}\right)  \log q_{k-1}} {\log q_{k-1}}=\lim_{k \to \infty} \left(1-\frac {1} {i(k)} \right) = 1.
$$
Thus, $\dimh{\Theta_{Q,S}}=1$.
\end{proof}

Clearly, every $1$-convergent basic sequence satisfies \refeq{growth}.  So, part (2) of \reft{mainthrm} follows by \refc{inRNisect} and \reft{dimhtq}.

\section{Further Remarks}

We observed after \refl{modS} that it was key to be able to approximate $\frac {E_n}{q_n}$ for $n \equiv 1 \pmod{S_j}$.  Part (2) of \reft{mainthrm} can be extended to a larger intersection of sets of the form $\DNQ \backslash \RNk{Q}{1}$ by estimating $\frac {E_n} {q_n}$ for $n \equiv r \pmod{S_j}$, $r=0,1,\cdots,S_j-1$.  Given $Q=Q_1 \LS{1} Q_2 \LS{2} Q_3 \cdots$,  define $Q_{j,k}=(q_{j,k,n})$ by
\begin{displaymath}
q_{j,k,n}=\left\{ \begin{array}{ll}
\prod_{j=1}^k q_j & \textrm{if $n=1$}\\
 & \\
\prod_{j=1}^s q_{s(n-1)+j+k} & \textrm{if $n>1$}
\end{array} \right. ,
\end{displaymath}
so $Q_j=Q_{j,0}$.  
%See Figure~2 for an illustration.  
With only small modifications of the preceeding proofs, we may conclude that
$$
\tq \subsetneq \bigcap_{j=1}^\infty \bigcap_{k=0}^{S_j-1} \DN{Q_{j,k}} \backslash \RNk{Q_{j,k}}{1}
\hbox{ and }
\dimh\left(\bigcap_{j=1}^\infty \bigcap_{k=0}^{S_j-1} \DN{Q_{j,k}} \backslash \RNk{Q_{j,k}}{1} \right)=1.
$$
The techniques introduced in this paper are unlikely to settle the following questions.
For an arbitrary countable collection of infinite in limit basic sequences $(Q_j)$, is it true that $\dimh\left( \RNisect\right)=1$?
A more difficult problem would be to construct an explicit example of a member of $\RNisect$.  The problem gets much harder if we loosen the restriction that $Q_j$ is infinite in limit.  In fact, it is still an open problem to construct an explicit example of a member of $\DNQ$ for an arbitrary $Q$.  See \cite{Laffer} for more information.

%\bibliographystyle{amsplain}

%\bibliography{mance} 
\input{NormalCantorSeries5.bbl}

%\end{thebibliography}
\end{document}

%% file: NormalCantorSeries5.bbl
\providecommand{\bysame}{\leavevmode\hbox to3em{\hrulefill}\thinspace}
\providecommand{\MR}{\relax\ifhmode\unskip\space\fi MR }
\providecommand{\MRhref}[2]{%
  \href{http://www.ams.org/mathscinet-getitem?mr=#1}{#2}
}
\providecommand{\href}[2]{#2}